\documentclass{amsart}

\usepackage{amsmath,amssymb,amscd}
\usepackage{hyperref}
\usepackage{ytableau}
\input{quivers_skewed.sty}

\ytableausetup{smalltableaux}

\newtheorem{theorem}{Theorem}[section]
\newtheorem{lemma}[theorem]{Lemma}

\newtheorem{proposition}[theorem]{Proposition}
\newtheorem{observation}[theorem]{Observation}

\theoremstyle{definition}
\newtheorem{definition}[theorem]{Definition}
\newtheorem{example}[theorem]{Example}

\theoremstyle{remark}

\numberwithin{equation}{section}

\newcommand{\GL}{\operatorname{GL}}
\newcommand{\SL}{\operatorname{SL}}

\newcommand{\module}{\operatorname{mod}}
\newcommand{\Rep}{\operatorname{Rep}}

\newcommand{\Hom}{\operatorname{Hom}}

\newcommand{\Aut}{\operatorname{Aut}}
\newcommand{\End}{\operatorname{End}}
\newcommand{\Ext}{\operatorname{Ext}}

\newcommand{\M}{\operatorname{M}}

\newcommand{\irr}{\operatorname{irr}}

\newcommand{\mb}[1]{\mathbb{#1}}

\newcommand{\mf}[1]{\mathfrak{#1}}
\newcommand{\op}[1]{\operatorname{#1}}
\newcommand{\SI}{\operatorname{SI}}
\newcommand{\SIR}{\operatorname{SIR}}

\newcommand{\innerprod}[1]{\langle#1\rangle}

\newcommand{\varsm}[1]{\left[\begin{smallmatrix}#1\end{smallmatrix}\right]}
\newcommand{\sm}[1]{\left(\begin{smallmatrix}#1\end{smallmatrix}\right)}

\begin{document}

\title{Constructing Coherently G-invariant Modules}
\author{Jiarui Fei}
\address{National Center for Theoretical Sciences, Taipei 10617, Taiwan}
\email{jrfei@ncts.ntu.edu.tw}
\thanks{}


\date{}
\dedicatory{}

\begin{abstract} Let $G$ be a reductive group acting on a path algebra $kQ$ as automorphisms. We assume that $G$ admits a graded polynomial representation theory, and the action is polynomial. We describe the quiver $Q_G$ of the smash product algebra $kQ\# k[\M_G]^*$, where $\M_G$ is the associated algebraic monoid of $G$. We use $Q_G$-representations to construct coherently $G$-invariant modules of $Q$.
As an application, we construct algebraic semi-invariants on the quiver representation spaces from those $G$-invariant modules.
\end{abstract}

\maketitle

\section*{Introduction}

Let $k$ be a field of characteristic 0, and $A$ be a finite-dimensional $k$-algebra with a finite group $G$ acting as automorphisms. Then we can form the skew group algebra $AG:=A\# k[G]$, which is a well-studied subject (e.g., \cite{RR}). $AG$ and $A$ have the same representation type and global dimension.
If the algebra is the path algebra of a finite quiver $Q$, and the action permutes the set of primitive idempotents and stabilizes the arrow span $kQ_1$, then the quiver $Q_G$ of $kQG$ can be explicitly described \cite{Dt, L} (see Section \ref{ss:Q_G_finite}).

A natural question is that if $G$ is a reductive group acting rationally on $A$ as automorphisms, what is a good analogue of the skew group algebra? One natural answer can be replacing the group algebra by the Hopf algebra $k[G]$, and forming the smash product $A\# k[G]^*$. However, the dual coordinate algebra $k[G]^*$ is not semisimple, and quite complicated in general. To describe the quiver of $kQ\# k[G]^*$ is a rather difficult task. So we consider the coordinate (bi)algebra of the {\em associated monoid} $k[\M_G]$ as an alternative. If $G$ admits a {\em graded polynomial representation theory} (Definition \ref{D:gradedpoly}), then $k[\M_G]^*$ is semisimple. So the price is that we need to restrict to a special class of reductive groups and require the action to be polynomial.
Then we can explicitly describe the quiver $Q_G$ of $kQ[\M_G]^*:=kQ\#k[\M_G]^*$. The quiver is possibly an infinite quiver, but each connected component is still finite-dimensional (Proposition \ref{P:component}). Theorem \ref{T:Q_G} is our first main result. The proof is similar to that in \cite{Dt}.

Let us come back to the finite group action. The action of $G$ on $A$ induces an action of $G$ on the category of (left) $A$-modules.
We write this induced action in the exponential form, that is, ${^g M}$ is the module $M$ with the action of $A$ twisted by $g$:
$$am=(g^{-1}a)m.$$
An $A$-module $M$ is called {\em $G$-invariant} if ${^g M}\cong M$ for any $g\in G$.
The restriction of an $AG$-module $M$ is a $G$-invariant $A$-module. The converse is almost true (Lemma \ref{L:AG-mod}) but false in general.
Those $kQ$-modules admitting a $kQG$-module structure are of our main interest. In fact, we only need something weaker called {\em proj-coherently $G$-invariant} (Definition \ref{D:G-invariant}). They contain all exceptional modules of $G$-stable dimension vectors (Observation \ref{ob:rigidSchur}).
To construct such $kQ$-modules, we need to concretely describe the Morita equivalence functor $kQ_G$-$\module\to kQG$-$\module$ composed with the restriction functor $kQG$-$\module\to kQ$-$\module$. This can be done as long as we can compute a complete set of primitive orthogonal idempotents of the group algebra $k[G]$ (see Section \ref{ss:functors}).

All above about finite group actions have analogue for our $kQ[\M_G]^*$. However, in this case $Q_G$ is possibly an infinite quiver, so it is quite impossible to completely describe the above functor. So we fix some connected component $Q_c$ of $Q_G$, then we can describe the analogous functor $kQ_c$-$\module\to kQ$-$\module$, provided we can compute a complete set of primitive orthogonal idempotents of some homogenous subalgebra $S_c$ of $k[\M_G]^*$ depending on $Q_c$. Such subalgebra $S_c$ is a finite direct product of {\em Schur algebras} of $G$.

Our motivation comes from constructing algebraic semi-invariants on the quiver representation spaces.
For some dimension vector $\alpha$, let $\Rep_\alpha(Q)$ be the space of all $\alpha$-dimensional representations of $Q$.
The product of general linear group $\GL_\alpha:=\prod_{v\in Q_0}\GL_{\alpha(v)}$ acts on $\Rep_\alpha(Q)$ by the natural base change.
In \cite{S1}, Schofield introduced for each representation $N\in\Rep_\beta(Q)$ with $\innerprod{\alpha,\beta}_Q=0$, a semi-invariant function $c_N\in k[\Rep_\alpha(Q)]$ for the above action. Here $\innerprod{-,-}_Q$ is the Euler form of $Q$.
In fact, $c_N$'s span the space of all semi-invariants of weight $\innerprod{-,\beta}_Q$ over the base field $k$ \cite{DW1,SV}.

The action of $G$ on $kQ$ induces $G$-actions on all representation spaces of $Q$. An easy observation is that if $N$ is proj-coherently $G$-invariant, then $c_N$ is also semi-invariant under $G$-action.
This observation allows us to construct new semi-invariants for the $\GL_\alpha \times G$-action on $k[\Rep_\alpha(Q)]$.
We are particularly interested in the setting of $n$-arrow Kronecker quivers $K_n$, where $G=\GL_n$ acting on the space of arrows.
The $(\alpha_1,\alpha_2)$-dimensional representation space of $K_n$ can be identified with the (tri-)tensor space $U^*\otimes V\otimes W^*$,
where $\dim (U,V,W) = (\alpha_1, \alpha_2, n)$.
To illustrate our method, we construct several such semi-invariants in Propositions \ref{P:32}, \ref{P:42}, \ref{P:33}, and \ref{P:233}. Proposition \ref{P:32} may be well-known, but we believe that the rest are new.

We hope to find the dimension of the linear span of semi-invariants of form $c_N$, where $N$ is coherently $G$-invariant of fixed dimension.
Theorem \ref{T:SIadj} converts this problem to a similar problem on the quiver $Q_c$. As we will see, when $Q_c$ is simple, the dimension can be easily calculated.

\subsection*{Notations and Conventions}
Our vectors are exclusively row vectors.
If an arrow of a quiver is denoted by a lowercase letter, then we use the same capital letter for its linear map of a representation.
For direct sum of $n$ copy of $M$, we write $nM$ instead of the traditional $M^{\oplus n}$. Unadorned $\Hom$ and $\otimes$ are all over the base field $k$, and the superscript $*$ is the trivial dual.

\section{Finite Group Action}
Let $k$ be a field of characteristic 0, and $G$ be a finite group acting on a $k$-algebra $A$ as automorphisms.
The group algebra $k[G]$ is a Hopf algebra with counit, comultiplication, and antipode defined by the linear extension of
$$\epsilon(g)=1,\quad \Delta(g)=g\otimes g,\quad S(g)=g^{-1}.$$
In this way, $A$ obtains a $k[G]$-module algebra structure.
For an element $c$ in a coalgebra, we use Sweedler's notation for comultiplication and coaction throughout. For example, we abbreviate $\Delta(c)=\sum_i c_{(0)}^{(i)}\otimes c_{(1)}^{(i)}$ to $\Delta(c)=\sum c_{(0)}\otimes c_{(1)}$.

\begin{definition} Let $B$ be a bialgebra. A (left) {\em $B$-module algebra} $A$ is an algebra which is a (left) module over $B$ such that
for any $b\in B, a,a'\in A$,
$$b 1_A = \epsilon(b) 1_A,\quad \text{ and }\quad b\cdot (aa')=\sum(b_{(0)}\cdot a)(b_{(1)}\cdot a').$$
The {\em smash product algebra} $A\#B$ is the vector space $A\otimes B$ with the product
$$(a\otimes b)(a'\otimes b'):= \sum_{}a(b_{(0)}\cdot a') \otimes b_{(1)}b'.$$
\end{definition}

When $B=k[G]$ is a group algebra, we may abuse of notation writing $a$ for $a\otimes 1_G$ and $b$ for $1_A\otimes b$.
In this context, $a\otimes b$ can be written as $ab$, and thus $A\# k[G]$ may be denoted by $AG$.

The action of $G$ on $A$ induces an action of $G$ on the category of (left) $A$-modules.
We write this induced action in the exponential form, that is, ${^g M}$ is the module $M$ with the action of $A$ twisted by $g$:
$$am=(g^{-1}a)m.$$
For morphisms $f\in\Hom_A(M,N)$, we check that the following defines a morphism ${^g f}\in\Hom_A({^g M},{^g N})$
$${^g f}(m)=f(m).$$
If $M$ is an $A$-module then $(AG)\otimes_A M$ is isomorphic as an $A$-module to $\bigoplus_{g\in G} {^g M}$, where the action of $G$ permutes the factors.

We observe that an $AG$-module $M$ is an $A$-module which is also a $G$-module, and such that
\begin{equation} \label{eq:GA} g(am) = (ga)(gm). \end{equation}

\begin{lemma} \label{L:AG-mod} An $A$-module $M$ admits a structure of an $AG$-module if and only if there is a family of isomorphisms $\{i_g: M\to {^g M}\}_{g\in G}$ satisfying ${^g i_h} i_g=i_{hg}$ for any $g,h\in G$.
\end{lemma}
\begin{proof}
If $M$ is an $AG$-module, then \eqref{eq:GA} says that the assignment $m\mapsto g^{-1}m$ defines an isomorphism $i_g:M\to {^g M}$.
Conversely, if we have a family of isomorphisms $\{i_g: M\to {^g M}\}_{g\in G}$ satisfying ${^g i_h} i_g=i_{hg}$ for any $g,h\in G$, then we can endow $M$ with a $G$-module structure as follows.
Note that $M$ and ${^g M}$ have the same underlying vector space on which ${^g i_h}=i_h$, so we can define a $G$-action on $M$ satisfying \eqref{eq:GA} by $g(m)=i_g(m)$.
\end{proof}

\begin{definition} \label{D:G-invariant} An $A$-module $M$ is called {\em $G$-invariant} if ${^g M}\cong M$ for any $g\in G$.
It is called {\em proj-coherently $G$-invariant} if there is a family of isomorphisms $\{i_g: M\to {^g M}\}_{g\in G}$ satisfying that $\forall g,h\in G, \exists c\in k^*$ such that ${^g i_h} i_g=c\cdot i_{hg}$.
It is called {\em coherently $G$-invariant} if it admits a $AG$-module structure.
A (coherently) $G$-invariant $A$-module is called {\em (coherently) $G$-indecomposable} if it is not a direct sum of two (coherently) $G$-invariant modules.
\end{definition}
For our main application on invariant theory, we are more interested in \text{(proj-)} coherently $G$-invariant modules.
In general, being $G$-invariant is strictly weaker than being coherently $G$-invariant.
However, when $G$ is cyclic and $A$ a path algebra, Gabriel \cite{G} proved that they are equivalent.

\begin{observation} \label{ob:rigidSchur} Let $A=kQ$ be the path algebra of a finite quiver $Q$, and $\alpha$ be a $G$-stable dimension vector.
\begin{enumerate}
\item A {\em rigid} $\alpha$-dimensional representation of $Q$ is $G$-invariant.
\item A $G$-invariant {\em Schur representation} of $Q$ is proj-coherently $G$-invariant.
\item If the cohomology group $H^2(G;k^*)$ vanishes, then proj-coherent is equivalent to coherent.
\end{enumerate} \end{observation}
\begin{proof} By definition $M$ is rigid if $\Ext_Q^1(M,M)=0$. So the orbit of $M$ is dense in the $\alpha$-dimensional representation space, which is irreducible.
But ${^g M}$ is rigid as well, so they have to be in the same orbit.

By definition, $M$ is Schur if $\Hom_Q(M,M)=k$. So the statement follows from the definition.

If $H^2(G;k^*)=0$, then every projective representation $G\to \GL_\alpha/k^*$ lifts to $G\to \GL_\alpha$. So we can modify each $i_g$ by some scalar factor such that ${^g i_h} i_g=i_{hg}$.
\end{proof}

\begin{definition} A dimension vector $\alpha$ of $Q$ is called a {\em $G$-root} if there is an $\alpha$-dimensional coherently $G$-indecomposable representation.
It is called a {\em strong $G$-root} if there is an indecomposable coherently $G$-invariant module.
\end{definition}

When $G$ is cyclic, all $G$-roots can be described in terms of the root system of associated valued quiver \cite{H}.
The following lemma is well-known.
\begin{lemma} \label{L:same} For any finite-dimensional algebra $A$, $AG$ and $A$ have the same global dimension and representation type.
\end{lemma}

\subsection{Description for $Q_G$} \label{ss:Q_G_finite}
By Lemma \ref{L:same}, $kQG$ is Morita equivalent to some hereditary algebra $kQ_G$.
There are algorithms to find the quiver $Q_G$ if the action permutes the set of primitive idempotents and stabilizes the arrow span $kQ_1$.
Let us recall the methods in \cite{Dt,L}.

Let $\tilde{Q}_0$ be a set of representatives of class of $Q_0$ under the action of $G$. For $u\in Q_0$, let $O_u$ be the orbit of $u$ and $G_u$ be the subgroup of $G$ stabilizing $e_u$.

For $(u,v)\in \tilde{Q}_0\times \tilde{Q}_0$, $G$ acts diagonally on the product of the orbits $O_u\times O_v$.
A set of representatives of the classes of this action will be denoted by $O_{uv}$.
We define $R_{uv}:=kQ(u,v)$ to be the vector space spanned by the arrows from $u$ to $v$. We regard $R_{uv}$ as a right $k[G_{uv}]:=k[G_u\cap G_v]$-module by restricting the action of $G$.

Let $\irr(G)$ denote the set of all irreducible representations of $G$. The vertex set of $Q_G$ is $$\bigcup_{v\in \tilde{Q}_0} \{u\}\times \irr(G_u).$$
The arrow set from $(u,\rho)$ to $(v,\sigma)$ is a basis of
$$\bigoplus_{(u',v')\in O_{uv}} \Hom_{k[G_{u'v'}]} (V_\rho, R_{u'v'}\otimes V_\sigma).$$
Here $\rho$ should be understood as a representation of $G_{u'}$ as follows.
Let $g_{uu'}$ be such that $g_{uu'} u = u'$, then $\rho(h)=\rho(g_{uu'}^{-1}hg_{uu'})$ for $h\in G_{u'}$. Similar identification makes $\sigma$ a representation of $G_{v'}$.

The proof uses the following idempotent $e$ of $kQG$, which will be used later.
Let $R$ be the maximal semisimple subalgebra of $kQ$.  Let $e_0=\sum_{u\in \tilde{Q}_0} e_u \in R\subset RG$. It is not hard to see that $e_0 (kQG) e_0$ is Morita equivalent to $kQG$, and $e_0 (RG) e_0\cong \prod_{u\in\tilde{Q}_0} k[G_u]$.
Since each $G_u$ is semi-simple, we can fix for each $u\in \tilde{Q}_0$ and $\rho\in\irr(G_u)$, a primitive idempotent $e_{u\rho}$ of $k[G_u]$ corresponding to $\rho$. Let
$$e=\sum_{u\in\tilde{Q}_0}\sum_{\rho\in\irr(G_u)} e_{u\rho}.$$
It is proved in \cite{Dt} that $e(kQG)e$ is a basic algebra Morita equivalent to $kQG$.

\subsection{Functors} \label{ss:functors}
Let $A:=kQ$ and $B:=kQ_G$. The functor $AG\otimes_A -$ has the restriction functor as its right adjoint.
The Morita equivalence functor $e(-)$ has $R_e:=\Hom_B(eAG,-)$ as its right adjoint.
So the composition $T:=e(AG\otimes_A -)$ has a right adjoint $R:=\op{res}\circ R_e$.
Note that $T$ is exact and preserves projective presentations, and thus $R$ preserves injective presentations.
Moreover, both $T$ and $R$ map semisimple modules to semisimple modules \cite[Theorem 1.3]{RR}.

The functor $AG\otimes_A -$ is also right adjoint to the restriction functor \cite[Theorem 1.2]{RR}. So $T$ also has a left adjoint $L:=\op{res}\circ AGe\otimes_B -$. However, in these notes we will exclusively work with the functor $R$.

Now we have the following diagram of functors
$$\vcenter{\xymatrix@=9ex{ & \module AG  \ar@<.5ex>[dl]^{\op{res}} \ar@<.5ex>[dr]^{e(-)} &  \\
\module A \ar@<.5ex>[ur]^{AG\otimes_A -}  & & \module B \ar@<.5ex>[ul]^{R_e} \ar[ll]^{R}
}}$$

By our construction, the functor $R$ sends the simple $S_{u\rho}$ corresponding to the vertex $e_{u\rho}$ to
the semisimple representation $\bigoplus_{v\in O_u} \dim(V_\rho) S_{v}$ of $Q$.
In this way, $R$ induces a linear map $r:K_0(B) \to K_0(A)$.
Since $R_e$ is an equivalence and preserves indecomposables, it follows that
\begin{proposition} $\alpha$ is a $G$-root if and only if there is a root $\beta$ of $Q_G$ such that $r(\beta)=\alpha$.
\end{proposition}

We want to give a concrete description for the functor $R$. To be more precise, we want to lift $R$ to a map between representation spaces of $Q_G$ and $Q$.
Clearly, such a description relies on the choice of a complete set of primitive orthogonal idempotents of $k[G_u]$ for each $u\in\tilde{Q}_0$.
In general, no explicit formula for primitive orthogonal idempotents in a finite group algebra is known. However, in many special cases, for example when the group is a symmetric group, a complete set of primitive orthogonal idempotents is given by the Young symmetrizers \eqref{eq:Young} \cite[9.3]{GW}.

Assume that we have got a complete set $I$ of primitive orthogonal idempotents of $k[G_u]$ for each $u\in\tilde{Q}_0$.
By Maschke's Theorem, $k[G_u]$ is a product of matrix algebras $\prod_{\rho\in\irr(G_u)} \End(V_\rho)$.
We can compute a standard basis $\{e^{ij}_{u\rho}\}$ of the matrix algebra $\End(V_\rho)$ such that $\{e^{ii}_{u\rho}\}\subset I$ and $e^{11}_{u\rho} = e_{u\rho}$.
We identify a basis of $\{e_{u\rho}^{1i} R_{uv} e_{v\sigma}^{j1}\}$ with some arrows from $(u,\rho)$ to $(v,\sigma)$, say $\{b_k\}_k$.
Now for each $a\in R_{uv}$, $\{e_{u\rho}^{ii} a e_{v\sigma}^{jj}\}$ is a linear combination of $e_{u\rho}^{i1} b_k e_{v\sigma}^{1j}$'s.
Say $e_{u\rho}^{ii} a e_{v\sigma}^{jj} = \sum c_k^{ij} e_{u\rho}^{i1} b_k e_{v\sigma}^{1j}$.

For any $N\in\Rep_\beta(Q_G)$, $M=R(N)\in\Rep_{r(\beta)}(Q)$ is the following representation. The vector space $M_u$ attached to the vertex $u$ is
$$M_u = \bigoplus_{\rho\in \irr(G_u)} d_\rho N_{u\rho}, \quad d_\rho=\dim (V_\rho).$$
Here, each copy of $N_{u\rho}$ corresponds to some $e^{ii}_{u\rho}$. Let us denote such a copy by $N_{u\rho}^i$.
The linear map from $N_{u\rho}^i$ to $N_{v\sigma}^j$ is given by substituting the arrows in $\sum_k c_k^{ij} b_k$ by corresponding matrices in $N$.
In particular, we see that such a lifting is a linear morphism $\Rep_{\beta}(Q_G)\to \Rep_{r(\beta)}(Q)$.

\begin{example} Let $S_n$ be the $n$-subspace quiver:
$$\Sn$$
The symmetric group $\mf{S}_n$ acts naturally on $S_n$.
In this way, we get an action of $\mf{S}_n$ on $kS_n$.
There are only two orbits on $Q_0$ represented by $n$ and $n+1$.
The stabilizers $G_n$ and $G_{n+1}$ are $\mf{S}_{n-1}$ and $\mf{S}_n$ respectively.
We have only one orbit in $O_n\times O_{n+1}$.
The irreducible representations of $S_n$ are indexed by partitions $\rho$, and primitive idempotents in $\End(V_\rho)$ can be labeled by Young tableaux $T$ of shape $\rho$:
\begin{equation} \label{eq:Young}
e_{T} = \kappa_\rho^{-1}\sum_{v \in V(T)}\sum_{h \in H(T)} \op{sgn}v \cdot vh.
\end{equation}
Here, $\kappa_\rho$ is the hook length of $\rho$, $V(T), H(T)$ are the vertical and horizontal subgroup corresponding to the Young tableaux $T$.
The number of arrows between $(n,\rho)$ and $(n+1,\sigma)$ is given by the multiplicity of $\rho$ in $\sigma$ restricted to $\mf{S}_{n-1}$.
This is equal to the Littlewood-Richardson coefficients $c_{\rho,(1)}^{\sigma}$, which can be computed by the Pieri rule.

For $n=4$, we get the following quiver for $B$
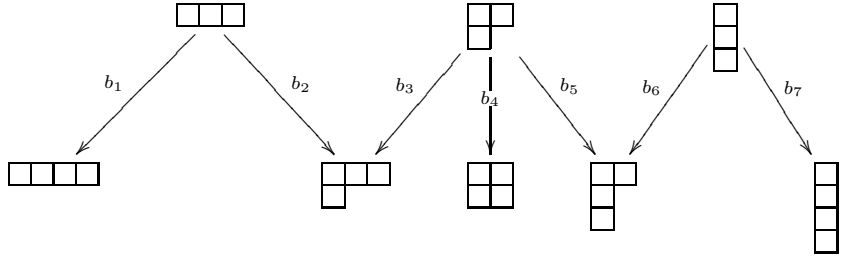
$$\skewSfour{\ydiagram{3}}{\ydiagram{2,1}}{\ydiagram{1,1,1}}{\ydiagram{4}}{\ydiagram{3,1}}{\ydiagram{2,2}}{\ydiagram{2,1,1}}{\ydiagram{1,1,1,1}}$$
The functor $R$ takes a representation of the above quiver to the following representation of $S_4$.
\begin{align*}
A_1&=\sm{B_1 & B_2 & -B_2 & -B_2 & 0 & 0 & 0 & 0 & 0 & 0 \\ 0 & B_3 & 0 & B_3 & B_4 & 0 & B_5 & -2B_5 & B_5 & 0 \\ 0 & 0 & B_3 & -B_3 & -B_4 & -B_4 & B_5 & B_5 & -2B_5 & 0 \\ 0 & 0 & 0 & 0 & 0 & 0 & B_6 & B_6 & B_6 & B_7}\\
A_2&=\sm{B_1 & B_2 & -B_2 & 3B_2 & 0 & 0 & 0 & 0 & 0 & 0 \\ 0 & B_3 & 0 & 0 & B_4 & B_4 & B_5 & 3B_5 & 0 & 0 \\ 0 & 0 & B_3 & 0 & -B_4 & 0 & B_5 & 0 & 3B_5 & 0 \\ 0 & 0 & 0 & 0 & 0 & 0 & B_6 & 0 & 0 & -B_7}\\
A_3&=\sm{B_1 & B_2 & 3B_2 & -B_2 & 0 & 0 & 0 & 0 & 0 & 0 \\ 0 & B_3 & 0 & 0 & -B_4 & -B_4 & -3B_5 & -B_5 & 0 & 0 \\ 0 & 0 & 0 & B_3 & 0 & B_4 & 0 & -B_5 & -3B_5 & 0 \\ 0 & 0 & 0 & 0 & 0 & 0 & 0 & -B_6 & 0 & -B_7}\\
A_4&=\sm{B_1 & -3B_2 & -B_2 & -B_2 & 0 & 0 & 0 & 0 & 0 & 0 \\ 0 & 0 & -B_3 & 0 & -B_4 & 0 & 3B_5 & 0 & B_5 & 0 \\ 0 & 0 & 0 & B_3 & 0 & -B_4 & 0 & 3B_5 & B_5 & 0 \\ 0 & 0 & 0 & 0 & 0 & 0 & 0 & 0 & B_6 & B_7}.
\end{align*}
\end{example}

\begin{example} The symmetric group $\mf{S}_n$ also acts naturally on the $n$-arrow Kronecker quiver $K_n$
$$\Kn$$
The $\mf{S}_n$-representation on arrows decomposes into the standard representation $(n-1,1)$ and the trivial representation,
so the number of arrows between $(1,\rho)$ and $(2,\sigma)$ is given by $g_{\rho,(n-1,1)}^{\sigma}+\delta_{\rho,\sigma}$.
Here, $g_{\rho,\pi}^{\sigma}$ is the Kronecker coefficient defined by $V_\rho \otimes V_\pi = \bigoplus_{\sigma} g_{\rho,\pi}^{\sigma} V_\sigma$.
Readers can verify the following quivers $Q_G$ together with the functor $R$ for $n=2,3$.
$$\skewKtwo{\ydiagram{1,1}}{\ydiagram{2}}{\ydiagram{1,1}}{\ydiagram{2}}$$
\begin{align*}
A_1=\sm{B_1 & B_2 \\ B_3 & B_4},
A_2=\sm{B_1 & -B_2 \\ -B_3 & B_4}.
\end{align*}

$$\skewKthree{\ydiagram{1,1,1}}{\ydiagram{2,1}}{\ydiagram{3}}{\ydiagram{1,1,1}}{\ydiagram{2,1}}{\ydiagram{3}}$$
\begin{align*}
A_1=\sm{B_1 & B_2 & B_2 & 0 \\ B_3 & \frac{B_4+B_5}{2} & \frac{B_5-B_4}{2} & B_6 \\ -B_3 & \frac{B_4-B_5}{2} & -\frac{B_4+B_5}{2} & B_6 \\ 0 & B_7 & -B_7 & B_8 },
A_2=\sm{B_1 & B_2 & -2B_2 & 0 \\ 0 & B_4 & 0 & -2B_6 \\ B_3 & \frac{B_5-B_4}{2} & -B_5 & B_6 \\ 0 & -B_7 & 0 & B_8 },
A_3=\sm{B_1 & -2B_2 & B_2 & 0 \\ -B_3 & B_5 & \frac{B_4-B_5}{2} & B_6 \\ 0 & 0 & -B_4 & -2B_6 \\ 0 & 0 & B_7 & B_8 }.
\end{align*}
\end{example}

\section{Schur Algebras of Reductive Monoid}
In this section, we recall several results from \cite{Dy}. We keep our assumption that the base field $k$ has characteristic 0. Let $\M_n$ be the affine algebraic monoid of $n\times n$ matrices over $k$.
We naturally identify the coordinate algebra $k[\M_n]$ with the polynomial algebra ${A}(n):=k[X]$, where $X=\{x_{ij}\}_{1\leq i\leq j \leq n}$. The polynomial algebra is graded by the usual monomial degree ${A}(n)=\bigoplus_{d\geq 0} {A}(n,d).$
Moreover, ${A}(n)$ is a bialgebra with coalgebra structure maps $\Delta,\epsilon$ defined by
$$\Delta(x_{ij})=\sum_{k=1}^n x_{ik}\otimes x_{kj},\quad \epsilon(x_{ij})=\delta_{ij}.$$
Thus each graded piece ${A}(n,d)$ is a subcoalgebra of ${A}(n)$. Hence, its linear dual ${S}(n,d):={A}(n,d)^*$ is a finite-dimensional $k$-algebra, known as the classical Schur algebra.

The coordinate algebra of the general linear group $\GL_n$ is the localization of ${A}(n)$ at the determinant function:
$k[\GL_n]=k[X,\det(X)^{-1}]$.
Let $G$ be a reductive closed subgroup of $\GL_n$. By a polynomial function on $G$, we mean the restriction to $G$ of a polynomial function in ${A}(n)$. We denote by $A(G)$ the algebra of polynomial function on $G$. It inherits bialgebra structure from ${A}(n)$.
By $A(G,d)$ we denote the image of $A(n,d)$ under the restriction map from $\GL_n$ to $G$. It is a subcoalgebra of $A(G)$.
We denote the linear dual of $A(G,d)$ by $S(G,d)$. It is a subalgebra of ${S}(n,d)$ because $A(G,d)$ is a quotient of ${A}(n,d)$.

\begin{definition} \label{D:gradedpoly} We say that $G$ admits a {\em graded polynomial representation theory} if the sum $\sum_{d\geq 0} A(G,d)$ is direct.
\end{definition}

A standard non-example is $\SL_n$ because $A(\SL_n,0)\cap A(\SL_n,n)\neq \emptyset$ due to the equation $\det(X)=1$. It is not hard to see that if $G$ contains the nonzero scalar matrices $cI$ of $\GL_n$, then $G$ admits a graded polynomial representation theory.
This includes, for example $\op{GSp}_n$ and $\op{GO}_n$, the groups of symplectic and orthogonal similitudes.
Proposition \ref{P:criterionGPRT} provides another criterion.

A finite dimensional (left) $G$-module $V$ is called rational if for some basis
$v_1,\dots,v_n$ of $V$ the corresponding coefficient functions $f_{ij}$, defined by the equations
$$g\cdot v_i=\sum_{j=1}^n f_{ij}(g)v_j$$
belong to $k[G]$.
We then have on $V$ the structure of a right $k[G]$-comodule via the structure map $\Delta_V: V\to V\otimes k[G]$, given by $\Delta_V(v_i)=\sum_{j=1}^n v_j\otimes f_{ij}$. It is well-known that there is an equivalence of categories between rational $G$-module and $k[G]$-comodules.
By a polynomial $G$-module we mean a vector space $V$ on which $G$ acts linearly with coefficient functions in $A(G)$.

\begin{proposition}{\cite[Propositions 1.3, 1.4]{Dy}} \label{P:catS} Suppose $G$ admits a graded polynomial representation theory. Every polynomial $G$-module has a direct sum decomposition into homogeneous polynomial representations. The category of homogeneous polynomial $G$-modules of degree $d$ is equivalent to the category of $S(G,d)$-modules.
\end{proposition}

We take $\M_G=\overline{G}$, the Zariski closure of $G$ in $\M_n$. Then $\M_G$ is a closed submonoid of $\M_n$ with $G$ as its group of units.
$\M_G$ is called the associated algebraic monoid of $G$. Let $I(\M_G)$ be the vanishing ideal of $\M_G$ in $\M_n$.

\begin{proposition}{\cite[Proposition 2.4]{Dy}} \label{P:criterionGPRT} $G$ admits a graded polynomial representation theory if and only if $I(\M_G)$ is homogeneous. In this case, we have a coalgebra isomorphism $A(G,d)\cong k[\M_G]_d$, so the algebra $S(G,d)$ consists of those elements in $S(n,d)$ vanishing on $I_d(\M_G)=A(n,d)\cap I(\M_G)$.
\end{proposition}

We provide a last point of view of $S(G,d)$ from the tensor power representations. Let $V$ be the ($n$-dimensional) natural $\op{M}_n$-representation. For any $d\in \mb{N}$, we have an action of $\op{M}_n$ on the $d$th tensor power of $V$, by $$A(v_1\otimes\cdots\otimes v_d)=Av_1\otimes\cdots\otimes Av_d.$$
Let $\phi_d$ be the corresponding representation $\op{M}_n\to \End(V^{\otimes d})$. It was proved by Schur \cite{M} that $S(n,d)=\op{span}(\phi_d(\GL_n))=\op{span}(\phi_d(\op{M}_n))=\End_{\mf{S}_d}(V^{\otimes d})$.

\begin{proposition}{\cite[Proposition 3.2]{Dy}} If $G$ admits a graded polynomial representation theory, then
$$S(G,d)=\op{span}(\phi_d(G))=\op{span}(\phi_d(\M_G)).$$
\end{proposition}

It is well-known that the semisimplicity of $\op{span}(\phi_d(G))$ is equivalent to complete reducibility of $V^{\otimes d}$ as $G$-module. So $\op{span}(\phi_d(G))$ is semisimple if $G$ is reductive.
We have the following monoid analogue of the Peter-Weyl theorem.

\begin{lemma} \label{L:k[M_G]} As $G$-bimodule algebras, $S(G,d)\cong\bigoplus_{\rho}\End(V_\rho)$, where $\rho$ runs through all irreducible degree $d$ polynomial representations of $G$.
So if $G$ admits a graded polynomial representation theory, then as $G$-bimodule algebras,
$$k[\M_G]^*\cong\prod_{\rho\in\irr(G)}\End(V_\rho),$$
where $\irr(G)$ is the set of all irreducible polynomial representations of $G$.
\end{lemma}
\begin{proof} The group $G\times G$ acts on $k[\op{M}_G]_d$ by the left and right translations.
Let $\rho$ be any degree $d$ polynomial representation of $G$.
We define $\varphi_\rho(v^*\otimes v) = \innerprod{v^*, \rho(g)v}$ for $g\in G, v^*\in V_\rho^*$, and $v\in V_\rho$. We extend the definition linearly to a map $V_\rho^*\otimes V_\rho \to k[\op{M}_G]_d$. It is easy to check that $\varphi_\rho$ is $G\times G$-equivariant.
Since $V_\rho^*\otimes V_\rho$ is an irreducible $G$-bimodule,
we have that $V_\rho^*\otimes V_\rho \hookrightarrow k[\op{M}_G]_d$ by Schur's lemma.
On the other hand, we can decompose $S(G,d)$ as a module over itself.
By Proposition \ref{P:catS}, only polynomial representations of $G$ can appear in the decomposition. We conclude that $S(G,d)\cong\bigoplus_{\rho}\End(V_\rho)$ as $G$-bimodules, where $\rho$ runs through all irreducible degree $d$ polynomial representations of $G$.
Finally, we need to show that the matrix multiplication in $\bigoplus_{\rho}\End(V_\rho)$ agrees with the multiplication in $S(G,d)$ so that $S:=S(G,d)$ is the $G$-bimodule algebra as required.
But this follows from
$S\cong \End_{S}(S,S)\cong \End_{S}(\bigoplus_{\rho}\End(V_\rho),\bigoplus_{\rho}\End(V_\rho))\cong \bigoplus_{\rho}\End(V_\rho)$ by Schur's lemma.
\end{proof}

Knowing that $S(G,d)$ is semisimple, it is an important problem to determine a complete set of primitive orthogonal idempotents. This can be a very hard problem in general, but for the classical Schur algebras $S(n,d)$, it is possible (especially when $d$ is small). Here are some simple examples, which will be used later.

Recall that the standard monomial basis of $A(n,d)$ is indexed by the {\em generalized permutations} $\sm{i_1&i_2&\cdots&i_d\\j_1&j_2&\cdots&j_d}$.
The pairs $\sm{i_k\\ j_k}$ are arranged in non-decreasing lexicographic order from left to right. In other words, the $i$'s are arranged in non-decreasing order, and the $j$'s corresponding to the same $i$ are in non-decreasing order.
We denote the corresponding dual basis in $S(n,d)$ by $\xi_{j_1j_2\cdots j_d}^{i_1i_2\cdots i_d}$. A nice combinatorial rule for multiplying such a basis is given in \cite{Mz}.

\begin{example} \label{ex:d2} Let $A=S(n,2)$. It has the following complete set of primitive orthogonal idempotents
\begin{align*}& \left\{ \frac{1}{2}\big(\xi_{ij}^{ij}-\xi_{ji}^{ij}\big) \right\}_{1\leq i < j\leq n} & \ydiagram{1,1}\\
& \left\{ \xi_{ii}^{ii},\frac{1}{2}\big(\xi_{ij}^{ij}+\xi_{ji}^{ij}\big) \right\}_{1\leq i < j\leq n}  & \ydiagram{2}
\end{align*}
The right column indicates the corresponding irreducible representations.
\end{example}
\begin{example} \label{ex:d3} Let $A=S(n,3)$. It has the following complete set of primitive orthogonal idempotents
\begin{align*}
& \left\{ \frac{1}{6}\sum_{\omega\in\mf{S}_3}\op{sgn}(\omega) \xi_{\omega(ijk)}^{ijk} \right\}  &  \ydiagram{1,1,1}\\
& \left\{ \frac{1}{3}\big(2\xi_{iij}^{iij}-\xi_{iji}^{iij}\big),
\frac{1}{3}\big(2\xi_{ijj}^{ijj}-\xi_{jij}^{ijj}\big),
\frac{1}{3} (\xi_{ikj}^{ijk}-\xi_{jki}^{ijk}+\xi_{ikj}^{ijk}-\xi_{jik}^{ijk}),
\frac{1}{3} (\xi_{ijk}^{ijk}-\xi_{ikj}^{ijk}+\xi_{jik}^{ijk}-\xi_{kij}^{ijk})\right\} & \ydiagram{2,1} \\
& \left\{ \xi_{iii}^{iii},\frac{1}{3}\big(\xi_{iij}^{iij}+\xi_{iji}^{iij}\big), \frac{1}{3}\big(\xi_{ijj}^{ijj}+\xi_{jij}^{ijj}\big), \frac{1}{6}\sum_{\omega\in\mf{S}_3}\xi_{\omega(ijk)}^{ijk} \right\} & \ydiagram{3}
\end{align*}
where $1\leq i<j<k\leq n$.
\end{example}

\section{Reductive Group Action}
\begin{definition}
Let $B$ be a $k$-bialgebra. A (right) {\em $B$-comodule algebra} $A$ is a $k$-algebra with a right $B$-comodule structure $\Delta_A:A\to A\otimes B$. We required $\Delta_A$ to be a $k$-algebra homomorphism. The {\em smash product algebra} $A\# B^*$ is by definition the vector space $A\otimes B^*$ with multiplication
$$(c\otimes h)(a\otimes f)=\sum ca_{(0)} \otimes (a_{(1)}\cdot h)f.$$
Here $a_{(1)}\cdot h$ is the usual (left) $B$-action on $B^*$, that is, $a_{(1)}h(b)=h(ba_{(1)})$.
\end{definition}
We observe that a left $A$-module $M$, which is also a right $B$-comodule $\Delta_M: M\to M\otimes B$ such that
$$\Delta_M(am)=\Delta_A(a)\Delta_M(m)$$ is a left $A\# B^*$-module, but not vice versa.
We may abuse of notation writing $a$ and $f$ for $a\otimes 1_{B^*}$ and $1_A\otimes f$.
If $\Delta_A(1)=1_A\otimes 1_{B}$, then we will write $af$ for $a\otimes f$ and $AB^*$ for $A\#B^*$ in this context.

Let $G$ be an infinite connected reductive group over $k$, and $\M_G$ be the associated algebraic monoid.
Since $G$ is algebraic, we will only consider rational action of $G$. In fact, we assume that $G$ acts
polynomially as automorphisms on some $k$-algebra $A$. Then $A$ becomes a $k[\M_G]$-comodule algebra.
As in the finite group case, we also have a $G$-action on the category of $A$-modules. We define (proj-coherently) $G$-invariant and $G$-indecomposable module as before.

The group $G$ can be naturally embedded into the dual coordinate algebra $k[\M_G]^*$.
For every $g\in G$, we define $\epsilon_g\in k[\M_G]^*$ as $\epsilon_g(f)=f(g)$.
Moreover, the embedding respects actions: $\epsilon_g(m)=\sum\epsilon_g(m_{(1)})m_{(0)}=\sum m_{(1)}(g)m_{(0)}=gm$.

\begin{proposition} \label{L:A[M_G]-mod} If $M$ is an $A\#k[\M_G]^*$-module, then $m\mapsto gm$ defines an $A$-module isomorphism $M\cong {^g M}$ for all $g\in G$.
\end{proposition}
\begin{proof}
We need to show for all $g\in G,a\in A,m\in M$ that
$$(1\otimes \epsilon_g)(a\otimes 1)(m)=g(am)=(ga)(gm).$$
For all $g\in G,a\in A,m\in M$, we have that
$$gm=\sum m_{(1)}(g) m_{(0)}, \text{ and }\ ga=\sum a_1(g) a_0.$$
Then \begin{align*}
(1\otimes \epsilon_g\cdot a\otimes 1)(m)&=\sum a_{(0)}\otimes (a_{(1)}\cdot \epsilon_g)(m)\\
&=\sum a_{(0)} \left( a_{(1)}\cdot \epsilon_g(m_{(1)}) \right) m_{(0)}\\
&=\sum a_{(0)} \epsilon_g(m_{(1)}a_{(1)}) m_{(0)}\\
&=\sum a_{(0)}m_{(1)}(g)a_{(1)}(g) m_{(0)}\\
&= (ga)(gm).
\end{align*} \end{proof}

Conversely, given a $G$-invariant $A$-module $M$, we assume that for each $g\in G$ we can fix an isomorphism $i_g:M\xrightarrow{} {^g M}$ such that
${^g i_h} i_g=i_{hg}$. Then we can define a $G$-action on $M$ by $g(m)=i_g(m)$. If such an action can be extended to $k[\M_G]^*$ (e.g., the action is polynomial), then we get an $A\#k[\M_G]^*$-module. To simplify the notation, we will write $A[\M_G]^*$ for $A\#k[\M_G]^*$.
Such a module as an $A$-module is called {\em coherently} $G$-invariant {\em in this context}. Under this definition, we also have the notion of (strong) $G$-root as in the finite group case.

Let $Q$ be a finite quiver without oriented cycles. The condition of no oriented cycles is not essential. But otherwise, we need to work with locally finite actions.
We keep our assumption that $G$ permutes the set of primitive orthogonal idempotents of $kQ$, and stabilizes the arrow span $kQ_1$.
Since the set of primitive orthogonal idempotents is finite but $G$ is infinite and connected, $G$ has to fix each idempotent.
In particular, $G$ is a reductive subgroup of $\Aut_1(Q):=\prod_{u,v \in Q_0} \GL(R_{uv})$, where $R_{uv}$ is the vector space spanned by arrows from $u$ to $v$.
From now on, we assume that $G$ admits a graded polynomial representation theory.

\subsection{Description of $Q_G$}
It turns out that $kQ[\M_G]^*$ is Morita equivalent to some hereditary algebra $kQ_G$.
The description is completely analogous to the one in Section \ref{ss:Q_G_finite}, except that $Q_G$ is possibly an infinite quiver.

Let $\irr(G)$ be the set of all polynomial representations of $G$.
The vertex set of $Q_G$ is $$\bigcup_{u\in Q_0} \{u\}\times \irr(G).$$
The arrow set from $(u,\rho)$ to $(v,\sigma)$ is a basis of
$$\Hom_{G} (V_\rho, R_{uv}\otimes V_\sigma).$$

\begin{theorem} \label{T:Q_G} Let $Q$ and $G$ be as above,
then $kQ[\M_G]^*$ is Morita equivalent to the path algebra $kQ_G$.
\end{theorem}
\begin{proof} Let $R$ be the (maximal semisimple) subalgebra of $kQ$ generated by the primitive orthogonal idempotents, and $R_1\subset kQ$ be the $R$-bimodule spanned by the arrows, so $kQ$ is the tensor algebra $T(R,R_1)$.

We fix for each $u\in Q_0$ and $\rho\in\irr(G)$, a primitive idempotent $e_{\rho}$ of $k[\M_G]^*$ corresponding to $\rho$ (see Lemma \ref{L:k[M_G]}).
Then $\{e_u\otimes e_\rho\}_{u\in Q_0,\rho\in\irr(G)}$ is a basic set of primitive orthogonal idempotents of $kQ[\M_G]^*$.
Let $e=\sum_{u\in Q_0,\rho\in\irr(G)} e_u\otimes e_\rho$, then
$$eR[\M_G]^*e=\prod_{u\in Q_0,\rho\in\irr(G)} ke_u\otimes e_\rho.$$

As $G$ stabilizes $R$ and $R_1$, it is easy to see that we have equivalence of categories
$$\module kQ[\M_G]^*\cong \module T(R[\M_G]^*,R_1[\M_G]^*) \cong \module T(eR[\M_G]^*e,eR_1[\M_G]^*e).$$
\begin{align*}
e_u\otimes e_\rho (R_1[\M_G]^*) e_v\otimes e_\sigma &= e_u\otimes e_\rho (R_1e_v\otimes k[\M_G]^* e_\sigma) \\
&= e_\rho (R_{uv}[\M_G]^* e_\sigma) \\
&=\Hom_k(k, e_{\rho} \left( k[\M_G]^*) (R_{uv}[\M_G]^* e_{\sigma}) \right)\\
&\cong \Hom_G(V_\rho, (R_{uv}[\M_G]^* e_{\sigma}))\\
&=\Hom_G(V_\rho, R_{uv}\otimes V_\sigma).
\end{align*}
\end{proof}

Since we are mainly interested in coherently $G$-indecomposable and indecomposable $G$-invariant modules, it is enough to focus on connected components of $Q_G$.
\begin{proposition} \label{P:component} If the quiver $Q$ is finite without oriented cycles, then each connected component of $Q_G$ is finite without oriented cycles.
\end{proposition}
\begin{proof}
We recall that $Q$ has no oriented cycles if and only if we can totally order the vertices of $Q$ such that $u<v$ if there is an arrow $u\to v$.
Now for a given component $Q_c$, we can totally order the vertices in $Q_c$ by $(u,\rho)<(v,\sigma)$ if $u<v$.
Note that $u<v$ is a necessary condition for having an arrow $(u,\rho)\to(v,\sigma)$.
Since $Q$ is finite, the linear span of arrows is a $G$-module of bounded degree.
So for each $e_u\otimes e_\rho$, it can be connected to only finitely many $e_v\otimes e_\sigma$ (by finitely many arrows).
But $Q$ has finitely many vertices in some total order, so the component containing $e_u\otimes e_\rho$ must be finite as well.
\end{proof}

We fix a connected component $Q_c$ of $Q_G$. Let $A:=kQ$ and $B:=kQ_c$.
Let $T_c: \module A\to \module B$ be the functor $e(A[\M_G]^*\otimes_A -)$ followed by the restriction to $Q_c$.
Let $R_c: \module B\to \module A$ be the functor $\Hom_{Q_c}(eA[\M_G]^*,-)$ followed by the restriction to $A$.
It is right adjoint to $T_c$, and can be lifted to an algebraic (in fact linear) morphism $\Rep_\beta(Q_c)\to \Rep_{r_c(\beta)}(Q)$ using a method similar to that in Section \ref{ss:Q_G_finite}.

\begin{example} For each finite quiver $Q$, we can associate a torus $T_1=(k^*)^{Q_1}$ acting naturally on $kQ_1$.
The irreducible representations of $T_1$ are all one-dimensional indexed by the weight lattice $\mathbb{Z}^{Q_1}$.
So the quiver $Q_G$ from our recipe is the {\em universal abelian covering quiver} of $Q$ (due to M. Reineke, see \cite[Section 3.1]{W}).
\end{example}

\begin{example} \label{ex:n=2}
Let $K_n$ be the $n$-arrow Kronecker quiver. The general linear group $\GL_n$ acts naturally on the arrow space of $K_n$.
This induces an action of $\GL_n$ on $kK_n$. The dimension of $\Hom_G(V_\rho, R_{uv}\otimes V_\sigma)$ is equal to the Littlewood-Richardson coefficients $c_{\sigma,(1)}^{\rho}$.

For any $n\geq 2$, the first component of $Q_G$ is always the following quiver.
$$\skewKthreetwo{\ydiagram{1,1}}{\ydiagram{2}}{\ydiagram{1}}$$
We can easily compute the functor $R_c$ using Example \ref{ex:d2}.
For $n=3$, the functor $R_c$ takes a representation of the above quiver to the following representation of $K_3$.
\begin{align*}
A_1=\sm{0&-B_1&0\\0&0&B_1\\0&0&0\\0&B_2&0\\0&0&B_2\\0&0&0\\B_2&0&0\\0&0&0\\0&0&0},
A_2=\sm{B_1&0&0\\0&0&0\\0&0&-B_1\\B_2&0&0\\0&0&0\\0&0&B_2\\0&0&0\\0&B_2&0\\0&0&0},
A_3=\sm{0&0&0\\-B_1&0&0\\0&B_1&0\\0&0&0\\B_2&0&0\\0&B_2&0\\0&0&0\\0&0&0\\0&0&B_2}.
\end{align*}

We observed that as the above situation, the matrices obtained are quite sparse.
We consider representation of them in three (one-dimensional) arrays, namely, the top row for values, the middle row for row numbers, and the bottom row for column numbers.
For example, the $A_1$ above is the block matrix $\sm{A_{1u}\\A_{1d}}$, where $A_{1u}=\varsm{-1&1\\1&2\\2&3}B_1$ and
$A_{1d}=\varsm{1&1&1\\1&2&4\\2&3&1}B_2$.

For $n=4$, the functor $R_c$ takes a representation of the above quiver to the representation $A_i=\sm{A_{iu}\\A_{id}}$ of $K_4$, where
\begin{align*}
A_{1u}&=\varsm{-1&1&1\\1&2&4\\2&3&4}B_1,&
A_{2u}&=\varsm{1&-1&-1\\1&3&5\\1&3&4}B_1, &
A_{3u}&=\varsm{-1&1&1\\2&3&6\\1&2&4}B_1, &
A_{4u}&=\varsm{-1&1&-1\\4&5&6\\1&2&3}B_1; \\
A_{1d}&=\varsm{1&1&1&1\\1&2&4&7\\2&3&4&1}B_2, &
A_{2d}&=\varsm{1&1&1&1\\1&3&5&8\\1&3&4&2}B_2, &
A_{3d}&=\varsm{1&1&1&1\\2&3&6&9\\1&2&4&3}B_2, &
A_{4d}&=\varsm{1&1&1&1\\4&5&6&10\\1&2&3&4}B_2.\end{align*}
\end{example}

\begin{example} \label{ex:n=3}
For $n\geq 3$, the second component of $Q_G$ is the following quiver
$$\skewKthreethree{\ydiagram{1,1,1}}{\ydiagram{2,1}}{\ydiagram{3}}{\ydiagram{1,1}}{\ydiagram{2}}
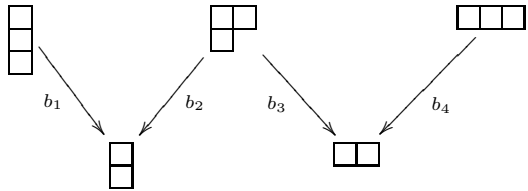
$$
\begin{align*}
\end{align*}
Using Example \ref{ex:d3}, we find that for $n=3$, the functor $R_c$ takes a representation of the above quiver to the following representation of $K_3$.
$$A_i=\begin{pmatrix} A_{iu}&0 \\ A_{il}&A_{ir} \\0&A_{id} \end{pmatrix}\qquad \text{where}$$
\begin{align*}
A_{1u}&=\varsm{1\\1\\3}B_1,& A_{2u}&=\varsm{1\\1\\2}B_1,& A_{3u}&=\varsm{1\\1\\1}B_1;\\
A_{1l}&=\varsm{-2&2&-1\\1&3&7\\1&2&3}B_2,&
A_{2l}&=\varsm{2&2&1&-1\\2&4&7&8\\1&3&2&2}B_2,&
A_{3l}&=\varsm{2&2&1\\5&6&8\\2&3&1}B_2;\\
A_{1r}&=\varsm{1&-2&1&2&\frac{1}{2}&-1\\1&2&3&5&7&8\\1&5&2&6&3&3}B_3,&
A_{2r}&=\varsm{-2&1&-1&-2&\frac{1}{2}&\frac{1}{2}\\1&2&4&6&7&8\\4&1&3&6&2&2}B_3,&
A_{3r}&=\varsm{-2&2&-1&1&-1&\frac{1}{2}\\3&4&5&6&7&8\\4&5&2&3&1&1}B_3;\\
A_{1d}&=\varsm{3&1&1&1&\frac{1}{2}&\frac{1}{8}\\1&4&5&6&8&10\\4&1&5&2&6&3}B_4,&
A_{2d}&=\varsm{3&1&1&\frac{1}{2}&\frac{1}{4}&\frac{1}{8}\\2&4&5&7&9&10\\5&4&1&3&6&2}B_4,&
A_{3d}&=\varsm{3&1&\frac{1}{2}&\frac{1}{2}&\frac{1}{4}&\frac{1}{8}\\3&6&7&8&9&10\\6&4&5&2&3&1}B_4.
\end{align*}
The next connect component is a Dynkin-$E_7$ for $n=3$ and extended-$E_7$ for $n>3$. Other components are all wild quivers.
\end{example}

\begin{example} \label{ex:S2}
As our last example, we still take the quiver $K_3$ but with a different action. We assume that the 3-dimensional space of arrows is the $\GL_2$-module $S^2(k^2)$.
Then the first component of $Q_G$ is  $$\skewKthreetwo{\ydiagram{2,1}}{\ydiagram{3}}{\ydiagram{1}}$$
The functor $R_c$ takes a representation of the above quiver to the following representation of $K_3$.
\begin{align*}
A_1=\sm{0&-B_1\\0&0\\3B_2&0\\0&0\\0&B_2\\0&0},
A_2=\sm{B_1&0\\0&B_1\\0&0\\0&0\\2B_2&0\\0&2B_2},
A_3=\sm{0&0\\-B_1&0\\0&0\\0&3B_2\\0&0\\B_2&0}.
\end{align*}
\end{example}

\section{Application to Tensor Invariants}
Let us briefly recall Schofield's semi-invariants of quiver representations \cite{S1}.
For a fixed dimension vector $\alpha$, the space of all $\alpha$-dimensional representations is
$$\Rep_\alpha(Q):=\bigoplus_{a\in Q_1}\Hom(k^{\alpha(ta)},k^{\alpha(ha)}).$$
The product of general linear group $\GL_\alpha:=\prod_{v\in Q_0}\GL_{\alpha(v)}$ acts on $\Rep_\alpha(Q)$ by the natural base change. This action has a {\em kernel}, which is the multi-diagonally embedded $k^*$.
For any {\em weight} $\sigma \in\mb{Z}^{Q_0}$, we can associate a character of $\GL_\alpha$ still denoted by $\sigma$
$$\left(g(v)\right)_{v\in Q_0}\mapsto\prod_{v\in Q_0} \big(\det g(v)\big)^{\sigma(v)}.$$
We define the subgroup $\GL_\alpha^\sigma$ to be the kernel of the character map. The semi-invariant ring $\SIR_\alpha^\sigma(Q):=k\left[\Rep_\alpha(Q)\right]^{\GL_\alpha^\sigma}$ of weight $\sigma$
is $\sigma$-graded: $\bigoplus_{n\geqslant 0} \SI_\alpha^{n\sigma}(Q)$, where
$$\SI_\alpha^{\sigma}(Q):=\{f\in k\left[\Rep_\alpha(Q)\right]\mid g(f)=\sigma(g)f, \forall g\in\GL_\alpha\}.$$

For any $N\in\Rep_{\beta}(Q)$, we choose some injective resolution of $N$
$$0\to N\to I_0 \to I_1\to 0,$$
and apply the functor $\Hom_Q(M,-)$ for $M\in\Rep_\alpha(Q)$
\begin{equation} \label{eq:canseq} \Hom_Q(M,N)\hookrightarrow\Hom_Q(M,I_0)\xrightarrow{\phi_M^N}\Hom_Q(M,I_1)\twoheadrightarrow\Ext_Q(M,N).\end{equation}

If $\innerprod{\alpha,\beta}_Q=0$, then $\phi_M^N$ is a square matrix. We fix a dual basis of $\Rep_\alpha(Q)$.
Following Schofield \cite{S1}, we define $c(M,N):=\det\phi_M^N$.
It is not hard to see that the definition only differs by a constant for other choices of the injective resolution of $N$.
In particular, we can take the canonical resolution or minimal resolution of $N$.
We can also define $c(M,N)$ using projective resolution of $M$.
Note that $c(M,N)\neq 0$ if and only if $\Hom_Q(M,N)=0$ or, equivalently, $\Ext_Q(M,N)=0$.
We denote $c_N:=c(-,N)$ and dually $c^M:=c(M,-)$.

It is proved in \cite{S1} that $c_N\in\SI_\alpha^{\sigma_\beta^\vee}(Q)$ for $\sigma_\beta^\vee=\innerprod{-,\beta}_Q$, and dually
$c^M\in\SI_\beta^{\sigma_\alpha}(Q)$ for $\sigma_\alpha=-\innerprod{\alpha,-}_Q$.
In fact, $c_N$'s (resp. $c^M$'s) span $\SI_\alpha^{\sigma_\beta^\vee}(Q)$ (resp. $\SI_\beta^{\sigma_\alpha}(Q)$) over the base field $k$ \cite{DW1,SV,DZ}.

Let $G$ be a finite group or an infinite connected reductive group acting polynomially on $kQ$ as automorphisms.
Such an action induces a rational action of $G$ on all representation spaces of $Q$.
We are interested in those semi-invariants which are also semi-invariant under the $G$-action.

\begin{observation} If $N$ is proj-coherently $G$-invariant, then $c_N$ is also semi-invariant under the $G$-action.
\end{observation}
\begin{proof} Since $N$ is proj-coherently $G$-invariant, there is some map $\varphi:G\to\GL_\alpha$ such that ${^g N}=\varphi(g)N$ and $\varphi$ descends to a representation $G\to\GL_\alpha/k^*$. Then
$$c_{^g N}(M)=c^M(\varphi(g)N)=\sigma_\alpha(\varphi(g))c^M(N)=(\sigma_\alpha \varphi)(g)c_N(M).$$
Since $\innerprod{\alpha,\beta}_Q=0$, $\sigma_\alpha\mid_{k^*}$ is trivial, so $\sigma_\alpha \varphi$ is a character of $G$.
In other words $c_N$ is semi-invariant under the $G$-action.
\end{proof}

This observation allows us to construct a lot of new semi-invariants for the $\GL_\alpha \times G$-action on $k\left[\Rep_\alpha(Q)\right]$.
According to Observation \ref{ob:rigidSchur}, any exceptional (=rigid Schur) representation is proj-coherently $G$-invariant.
Actually we conjecture that they are all coherently $G$-invariant.
The dimension of such a representation is a {\em real Schur root} $\gamma$ of the quiver. Moreover, for any two general representations $N_1,N_2\in\Rep_\gamma(Q)$, $c_{N_1}$ is a multiple of $c_{N_2}$.
In this sense, we will treat these semi-invariants as trivial, and avoid them later.

We are particularly interested in applying the method to construct the semi-invariants of (tri)-tensors.
By a (tri-)tensor of vector spaces $(U,V,W)$, we mean the vector space $U^*\otimes V\otimes W^*$. The product of special linear groups $SL:=\SL(U)\times \SL(V)\times \SL(W)$ acts naturally on it. We are interested in the invariants in $k[U^*\otimes V\otimes W^*]$ for this action.
The tensor space can be identified with the $(\alpha_1,\alpha_2)$-dimensional representation space of the $n$-arrow Kronecker quiver $K_n$, where $\dim U = \alpha_1, \dim V=\alpha_2,$ and $\dim W =n$. In this context, $G=\GL(W)$.

It follows from Example \ref{ex:n=2} that
\begin{proposition} \label{P:32} For general square matrices $B_1,B_2$, we define the representations $N_1, N_2$ of $K_3$
\begin{align*}
N_1(a_1)&=\sm{0&-B_1&0\\0&0&B_1\\0&0&0},&
N_1(a_2)&=\sm{B_1&0&0\\0&0&0\\0&0&-B_1},&
N_1(a_3)&=\sm{0&0&0\\-B_1&0&0\\0&B_1&0},&\\
N_2(a_1)&=\sm{0&B_2&0\\0&0&B_2\\0&0&0\\B_2&0&0\\0&0&0\\0&0&0},&
N_2(a_2)&=\sm{B_2&0&0\\0&0&0\\0&0&B_2\\0&0&0\\0&B_2&0\\0&0&0},&
N_2(a_3)&=\sm{0&0&0\\B_2&0&0\\0&B_2&0\\0&0&0\\0&0&0\\0&0&B_2}.&
\end{align*}
Then $c_{N_1}$ (resp. $c_{N_2}$) is a semi-invariant function for the tensor of size $a\times 2a\times 3$ (resp. $a\times a\times 3$).
\end{proposition}

\begin{proposition} \label{P:42}
For general square matrices $B_1,B_2$, we define the representations $N_1,N_2$ of $K_4$ by $A_{iu},A_{id}$ as in Example \ref{ex:n=2},
then
$c_{N_1}$ (resp. $c_{N_2}$) is a semi-invariant function for the tensor of size $2a\times 5a\times 4$ (resp. $2a\times 3a\times 4$).
\end{proposition}

\begin{proposition} \label{P:33}
For general square matrices $B_1,B_2,B_3,B_4$, we define the representations $N_1,N_2$ of $K_3$ by $A_{ir},A_{id}$ as in Example \ref{ex:n=3},
then
$c_{N_1}$ (resp. $c_{N_2}$) is a semi-invariant function for the tensor of size $3a\times 5a\times 3$ (resp. $3a\times 4a\times 3$).

We define the representation $N_3,N_4,N_5$ of $K_3$ by
\begin{align*} N_3(a_i)=\sm{A_{il}&A_{ir}},
N_4(a_i)=\sm{A_{il}&A_{ir} \\0&A_{id}},
N_5(a_i)=\sm{A_{iu}&0 \\ A_{il}&A_{ir} \\0&A_{id}}.
\end{align*}
Then
$c_{N_3}$ (resp. $c_{N_4},c_{N_5}$) is a semi-invariant function for the tensor of size $9a\times 19a\times 3$ (resp. $8a\times 9a\times 3$, $a\times a\times 3$).
\end{proposition}

We remark that our construction also applies to the case when the third factor $W$ is another representation of $\GL(W)$.
\begin{proposition} \label{P:233} For general square matrices $B_1,B_2$, we define the representations $N_1, N_2$ of $K_3$ (see Example \ref{ex:S2})
\begin{align*}
N_1(a_1)&=\sm{0&B_1\\0&0},&
N_1(a_2)&=\sm{B_1&0\\0&B_1},&
N_1(a_3)&=\sm{0&0\\-B_1&0},&\\
N_2(a_1)&=\sm{3B_2&0\\0&0\\0&B_2\\0&0},&
N_2(a_2)&=\sm{0&0\\0&0\\2B_2&0\\0&2B_2},&
N_2(a_3)&=\sm{0&0\\0&3B_2\\0&0\\B_2&0}.&
\end{align*}
Then $c_{N_1}$ (resp. $c_{N_2}$) is a semi-invariant function in $k[U^*\otimes V\otimes S^2(W)^*]$ for $\dim(U,V,W)=(a,2a,2)$ (resp. $(a,a,2)$).
\end{proposition}

Fix a component $Q_c$ of $Q_G$. Let $\SI_\alpha^{\sigma_{R_c(\beta)}^\vee}(Q)$ be the vector space spanned by semi-invariants on $\Rep_\alpha(Q)$ of form $c_{R_c(N)}$ for $N\in\Rep_\beta(Q_c)$. On the other hand, we can restrict a semi-invariant $c_N\in \SI_{r_c(\beta)}^{\sigma_\alpha}(Q)$ on the subvariety $R_c(\Rep_\beta(Q_c))$.
We denote the linear span of these restricted semi-invariants by $\SI_{R_c(\beta)}^{\sigma_\alpha}(Q)$.
Similarly to \cite[Corollary 1]{DW1}, we have the following reciprocity property
\begin{proposition} \label{P:reciprocity} $\dim\SI_\alpha^{\sigma_{R_c(\beta)}^\vee}(Q)=\dim\SI_{R_c(\beta)}^{\sigma_\alpha}(Q).$
\end{proposition}

In general, we do not know a simple method to compute the dimension of $\SI_\alpha^{\sigma_{R_c(\beta)}^\vee}(Q)$.
Sometimes, it is easier to perform computation on $Q_c$ using the theorem below.
To prove the theorem, we need some construction related to the functor $T_c$.
We can algebraically lift $T_c$ as we did for $R_c$. Moreover, the lifting can be constructed at the level of morphisms.
For our purpose, we only state such a lifting for morphisms between projectives.
It is enough to do this for $P_{v} \xrightarrow{a} P_{u}$, where $P_u, P_v$ are indecomposable projective representations corresponding to $u,v\in Q_0$, and $a$ is an arrow $u\to v$.
The construction will depend on the lifting of $R_c$.
Recall that a lifting of $R_c$ maps a representation $N$ of $Q_c$ to a representation $M$ of $Q$ as follows.
The vector space $M_u$ attached to the vertex $u$ is
$M_u = \bigoplus_{\rho\in Q_c} \dim (V_\rho) N_{u\rho}.$
Here, by $\rho\in Q_c$ we mean that there is an idempotent in $Q_c$ corresponding to the irreducible representation $\rho$.
The linear map from the $i$-th copy of $N_{u\rho}$ to $j$-th copy of $N_{v\sigma}$ is given by substituting the arrows $b_k$ in certain linear combination $\sum_k c_k^{ij} b_k$ by corresponding matrices in $N$.

Now we let $T_c$ send $P_u$ to
$T_c(P_u) = \bigoplus_{\rho\in Q_c} \dim (V_\rho) P_{u\rho},$
and send the morphism $P_{v} \xrightarrow{a} P_{u}$ to a matrix with $\sum_k c_k^{ij} b_k$ as the $ij$-th entry.
We see from the construction that such a lifting is not only algebraic but also compatible with the adjunction in the sense that
$\Hom_Q(P_u,R_c(N))$ can be naturally identified with $\Hom_{Q_c}(T_c(P_u),N)$ such that the diagram commutes
$$\xymatrix @C=3pc {
\Hom_Q(P_u,R_c(N)) \ar[rr]^{\Hom_Q(a,R_c(N))}\ar@{=}[d]^{} && \Hom_Q(P_v,R_c(N)) \ar@{=}[d]^{}\\
\Hom_{Q_c}(T_c(P_u),N) \ar[rr]^{\Hom_{Q_c}(T_c(a),N)} && \Hom_{Q_c}(T_c(P_v),N).
}$$
We remind readers that a morphism $P_1\xrightarrow{f} P_0$ can be represented by a matrix whose entries are linear combination of paths, and applying $\Hom_Q(-,N)$ to this morphism is nothing but substituting arrows in the matrix by corresponding matrix representation in $N$.

Let $\SI_\beta^{\sigma_{T_c(\alpha)}}(Q_c)$ be the vector space spanned by semi-invariants on $\Rep_\beta(Q_c)$ of form $c^{T_c(M)}$ for $M\in\Rep_\alpha(Q)$.

By Proposition \ref{P:reciprocity}, $\dim \SI_\beta^{\sigma_{T_c(\alpha)}}(Q_c)=\dim \SI_{T_c(\alpha)}^{\sigma_\beta^\vee}(Q_c)$, where $\SI_{T_c(\alpha)}^{\sigma_\beta^\vee}(Q_c)$ is the space of restricted semi-invariants on the subvariety $T_c(\Rep_\alpha(Q))$.
\begin{theorem} \label{T:SIadj} $\dim \SI_\alpha^{\sigma_{R_c(\beta)}^\vee}(Q) = \dim \SI_\beta^{\sigma_{T_c(\alpha)}}(Q_c)$.
\end{theorem}

\begin{proof} For any two representations $M\in\Rep_\alpha(Q), N\in\Rep_\beta(Q_c)$, we
take the canonical resolution $0\to P_1\to P_0 \to M\to 0$, and apply the functor $\Hom_{Q}(-,R_c(N))$, then we get
$$\xymatrix @C=1pc {
\Hom_Q(M,R_c(N)) \ar@{^(->}[r]\ar@{=}[d] & \Hom_Q(P_0,R_c(N)) \ar[rr]^{\phi_M^{R_c(N)}}\ar@{=}[d]^{} && \Hom_Q(P_1,R_c(N)) \ar@{>>}[r] \ar@{=}[d]^{} & \Ext_Q(M,R_c(N))\ar@{=}[d]\\
\Hom_{Q_c}(T_c(M),N) \ar@{^(->}[r] & \Hom_{Q_c}(T_c(P_0),N) \ar[rr]^{\phi_{T_c(M)}^N} && \Hom_{Q_c}(T_c(P_1),N) \ar@{>>}[r] & \Ext_{Q_c}(T_c(M),N).
}$$
The lower row is due to the adjunction.
Since $T_c$ is exact and preserves projectives, $0\to T_c(P_1)\to T_c(P_0)\to T_c(M)\to 0$ is in fact a projective resolution of $T_c(M)$.
By our construction of $T_c$, we conclude that
$$c(M,R_c(N))=\det \phi_M^{R_c(N)} = \det \phi_{T_c(M)}^N=c(T_c(M),N).$$
We view both functions as regular functions on $\Rep_\beta(Q_c)$, so $\{c(M,R_c(N))\}_M$ and $\{c(T_c(M),N)\}_M$ span the same subspace.
By Proposition \ref{P:reciprocity}, the dimension of former span is equal to $\dim \SI_\alpha^{\sigma_{R_c(\beta)}^\vee}(Q)$, and the dimension of the latter span is by definition $\dim \SI_\beta^{\sigma_{T_c(\alpha)}}(Q_c)$.
Therefore, $\dim \SI_\alpha^{\sigma_{R_c(\beta)}^\vee}(Q) = \dim \SI_\beta^{\sigma_{T_c(\alpha)}}(Q_c)$.
\end{proof}

\ytableausetup{mathmode, boxsize=.3em}

As an example, let us compute the dimension of $\SI_{(1,2)}^{\sigma_{R_c(1,0,1)}^\vee}(K_3)$ in Proposition \ref{P:32}.
It is enough to compute the dimension of $\SI_{(1,0,1)}^{\sigma_{T_c(1,2)}}(Q_c)$.
This $Q_c$ is a finite type quiver, so the dimension of $\SI_{(1,0,1)}^{\sigma_{T_c(1,2)}}(Q_c)$ is at most one.
A general representation $M$ in $\Rep_{(1,2)}(K_3)$ has resolution $0\to P_2\xrightarrow{k_1a_1+k_2a_2+k_3a_3} P_1 \to M\to 0$, then
$$0\to T_c(P_2)=3P_{\ydiagram{1}}\xrightarrow{\sm{k_2b_1&-k_1b_1&0\\-k_3b_1&0&k_1b_1\\0&k_3b_1&-k_2b_1}} 3P_{\ydiagram{1,1}}=T_c(P_1) \to T_c(M)\to 0.$$
Now it is not hard to see that $T_c(M)$ decomposes as $3(M_1\oplus M_2)$, where $M_1$ (resp. $M_2$) is a general representation of dimension $(0,1,1)$ (resp. $(1,1,1)$). So we see that $\Hom_{Q_c}(T_c(M),N)=0$ for general $N\in\Rep_{(1,0,1)}(Q_c)$, and thus $\dim \SI_{(1,2)}^{\sigma_{R_c(1,0,1)}^\vee}(K_3)=1$.
In fact, $\dim \SI_{(a,2a)}^{\sigma_{R_c(1,0,1)}^\vee}(K_3)=1$ for all $a\in\mathbb{N}$.

We checked that the spaces of semi-invariants of fixed weight in Propositions \ref{P:32}, \ref{P:42}, \ref{P:33}, and \ref{P:233} are all one-dimensional by hand and by computer. This theorem also tells us that to construct nontrivial semi-invariants, it is enough to use those {\em stable} representation of $Q_c$ in the sense of \cite{K}.

\section*{Acknowledgement}
The author wants to thank Liping Li for introducing him to think in $EI$-categories. He would also like to thank Professor Doty for some discussion on Schur algebras. He is especially grateful to the anonymous referee for several important suggestions and corrections.

\bibliographystyle{amsplain}

\end{document}